\documentclass[reqno]{amsart}

\usepackage[latin1]{inputenc}
\usepackage{amssymb}
\usepackage{graphicx}
\usepackage{amscd}
\usepackage[hidelinks]{hyperref}
\usepackage{color}
\usepackage{float}
\usepackage{graphics,amsmath,amssymb}
\usepackage{amsthm}
\usepackage{amsfonts}
\usepackage{latexsym}
\usepackage{epsf}
\usepackage{enumerate}
\usepackage{xifthen}
\usepackage{mathrsfs}
\usepackage{dsfont}
\usepackage{makecell}
\usepackage[FIGTOPCAP]{subfigure}
\usepackage{amsmath}
\allowdisplaybreaks[4]
\usepackage{listings}
\usepackage{etoolbox}
\usepackage{fancyhdr}

 \newtheoremstyle{mytheorem}
 {3pt}
 {3pt}
 {\slshape}
 {}
 {\bfseries}
 {.}
 { }
 {}

\numberwithin{equation}{section}

\theoremstyle{theorem}
\newtheorem{theorem}{Theorem}[section]

\newtheorem{lemma}[theorem]{Lemma}

\theoremstyle{definition}

\newtheorem{example}{Example}[section]

\newtheorem{remark}{Remark}[section]

\newcommand{\Keywords}[1]{\ifthenelse{\isempty{#1}}{}{\smallskip \smallskip \noindent \textbf{Keywords}. #1}}
\newcommand{\MSC}[2][2010]{\ifthenelse{\isempty{#2}}{}{\smallskip \smallskip \noindent \textbf{#1MSC}. #2}}
\newcommand{\abstractnote}[1]{\ifthenelse{\isempty{#1}}{}{\smallskip \smallskip \noindent \textsuperscript{\dag}#1}}

\makeatletter
\def\specialsection{\@startsection{section}{1}%
  \z@{\linespacing\@plus\linespacing}{.5\linespacing}%
  {\normalfont}}
\def\section{\@startsection{section}{1}%
  \z@{.7\linespacing\@plus\linespacing}{.5\linespacing}%
  {\normalfont\scshape}}
\patchcmd{\@settitle}{\uppercasenonmath\@title}{\Large\boldmath}{}{}
\patchcmd{\@settitle}{\begin{center}}{\begin{flushleft}}{}{}
\patchcmd{\@settitle}{\end{center}}{\end{flushleft}}{}{}
\patchcmd{\@setauthors}{\MakeUppercase}{\normalsize}{}{}
\patchcmd{\@setauthors}{\centering}{\raggedright}{}{}
\patchcmd{\section}{\scshape}{\large\bfseries\boldmath}{}{}
\patchcmd{\subsection}{\bfseries}{\bfseries\boldmath}{}{}
\renewcommand{\@secnumfont}{\bfseries}
\patchcmd{\@startsection}{\@afterindenttrue}{\@afterindentfalse}{}{}
\patchcmd{\abstract}{\leftmargin3pc}{\leftmargin1pc}{}{}

\def\maketitle{\par
  \@topnum\z@ 
  \@setcopyright
  \thispagestyle{empty}
  \ifx\@empty\shortauthors \let\shortauthors\shorttitle
  \else \andify\shortauthors
  \fi
  \@maketitle@hook
  \begingroup
  \@maketitle
  \toks@\@xp{\shortauthors}\@temptokena\@xp{\shorttitle}%
  \toks4{\def\\{ \ignorespaces}}
  \edef\@tempa{%
    \@nx\markboth{\the\toks4
      \@nx\MakeUppercase{\the\toks@}}{\the\@temptokena}}%
  \@tempa
  \endgroup
  \c@footnote\z@
  \@cleartopmattertags
}
\makeatother

\newcommand{\eoc}{\mathrm{eoc}}
\newcommand{\eo}{\mathrm{eo}}
\newcommand{\oeo}{\overline{\mathrm{eo}}}

\newcommand{\cp}{\mathcal{P}}
\newcommand{\cq}{\mathcal{Q}}
\newcommand{\co}{\mathcal{O}}
\newcommand{\ce}{\mathcal{E}}
\newcommand{\ceo}{\mathcal{EO}}

\title[On partitions with even parts below odd parts]{On partitions with even parts below odd parts}

\author[S. Chern]{Shane Chern}
\address{Department of Mathematics, The Pennsylvania State University, University Park, PA 16802, USA}
\email{shanechern@psu.edu; chenxiaohang92@gmail.com}

\date{}

\begin{document}

%

\maketitle

\begin{abstract}

Recently, Andrews gave a detailed study of partitions with even parts below odd parts in which only the largest even part appears an odd number of times. In this paper, we provide a combinatorial proof of the generating function identity of such partitions. We also have a further investigation on the largest even part. Finally, we give an interesting weighted overpartition generalization.


\MSC{Primary 05A17; Secondary 11P83.}
\end{abstract}

\section{Introduction}\label{sect:intro}

Recently, Andrews \cite{And2017} gave a detailed study of partitions with even parts below odd parts in which only the largest even part appears an odd number of times (Here we allow partitions
with no even parts where we tacitly assume $0$ is the largest even and appears once). Let $\ceo^*$ be the set of such partitions. The most interesting property is the following generating function identity
\begin{equation}\label{eq:gf1}
\sum_{\pi \in \ceo^*} q^{|\pi|}=\frac{(q^4;q^4)_\infty}{(q^2;q^4)_\infty^2},
\end{equation}
where $|\pi|$ denotes the sum of all parts of $\pi$ and
\begin{align*}
(a;q)_n&:=\prod_{k=0}^{n-1} (1-a q^k),\\
(a;q)_\infty&:=\prod_{k=0}^{\infty} (1-a q^k).
\end{align*}

Andrews further defined the even-odd crank of $\pi \in \ceo^*$ by
$$\eoc(\pi):=\text{largest even part} - \sharp(\text{odd parts of $\pi$}).$$
He proved
\begin{equation}\label{eq:gf2}
\sum_{\pi \in \ceo^*} z^{\eoc(\pi)}q^{|\pi|}=\frac{(q^4;q^4)_\infty}{(z^2q^2;q^4)_\infty(z^{-2}q^2;q^4)_\infty},
\end{equation}
and obtained a mod $5$ congruence for this partition function.

At the end of his paper, Andrews asked for a combinatorial interpretation of \eqref{eq:gf2}. This stimulates the first topic of this paper. We will also have a further investigation on the largest even part. At last, we study an interesting weighted overpartition generalization of Andrews' partition function.

\section{A combinatorial proof of Andrews' generating function identity}\label{sect:gf}

Throughout this section, let $k$ and $r$ be nonnegative integers. We begin with some definitions.

Let $\cp$ be the set of ordinary partitions. Let $\cp_k$ be the set of partitions with largest part $=k$ and $\cp_{\le k}$ be the set of partitions with largest part $\le k$.

Let $\co_k^*$ be the set of odd partitions with exactly $2k$ parts where each different part appears an even number of times. Let $\ce^*$ be the set of even partitions where each different part appears an even number of times.

Let $\cq_{k,r}$ be the set of partitions with a $(k+r+1)\times k$ Durfee rectangle with the largest part below the Durfee rectangle $< k+r+1$. (Hence $\cq_{0,r}$ denotes the set of partitions with largest part $\le r$). Here the $(k+r+1)\times k$ Durfee rectangle of a partition is the largest rectangle in the Young diagram of the partition with the length of the rectangle being $r+1$ more than the width.

Let $\ceo_{k,r}^*$ be the set of partitions in $\ceo^*$ (defined in Sect.~\ref{sect:intro}) with $2k$ odd parts and largest even part $=2k+2r$ (and hence even-odd crank $=2r$).

Given any partition $\pi$, we write it in weakly decreasing order $(\pi_1,\pi_2,\ldots,\pi_\ell)$, where $\pi_1$ is the largest part and $\ell=\ell(\pi)$ counts the number of parts of $\pi$.

We first prove an interesting identity.

\begin{lemma}\label{le:1}
\begin{equation}
\sum_{k\ge 0} \frac{q^k}{(q;q)_k (q;q)_{k+r}} = \frac{1}{(q;q)_\infty} \sum_{k\ge 0} \frac{q^{k(k+r+1)}}{(q;q)_k (q;q)_{k+r}}.
\end{equation}
\end{lemma}

An analytic proof of this identity can be obtained by taking $c\to q^{r+1}$ and $z\to q$ and letting $a$ and $b$ tend to $0$ in Heine's third transformation (cf.~\cite[Eq.~(17.6.8)]{And2010})
$$\sum_{n\ge 0} \frac{(a;q)_n (b;q)_n z^n}{(q;q)_n (c;q)_n}=\frac{(abz/c;q)_\infty}{(z;q)_\infty}\sum_{n\ge 0} \frac{(c/a;q)_n (c/b;q)_n }{(q;q)_n (c;q)_n }\left(\frac{abz}{c}\right)^n.$$
We now give a combinatorial proof, which is motivated by \cite{Kim2010}.

\begin{proof}[Combinatorial proof of Lemma \ref{le:1}]
We first notice that
$$\sum_{k\ge 0} \frac{q^k}{(q;q)_k (q;q)_{k+r}}=\sum_{(\lambda,\pi)\in\cup_{k\ge 0}\left(\cp_k\times \cp_{\le k+r}\right)} q^{|\lambda|+|\pi|}$$
and
$$\frac{1}{(q;q)_\infty} \sum_{k\ge 0} \frac{q^{k(k+r+1)}}{(q;q)_k (q;q)_{k+r}} =\sum_{(\mu,\nu)\in\cup_{k\ge 0}\left(\cp \times \cq_{k,r}\right)} q^{|\mu|+|\nu|}.$$
Hence we only need to construct a bijection $\phi$ between $\cup_{k\ge 0}\left(\cp_k\times \cp_{\le k+r}\right)$ and $\cup_{k\ge 0}\left(\cp \times \cq_{k,r}\right)$.

Let $(\lambda,\pi)\in \cp_k\times \cp_{\le k+r}$ for some $k\ge 0$. We write $\phi((\lambda,\pi))=(\mu,\nu)$ for convenience. Now we seperate $\cp_k\times \cp_{\le k+r}$ into three disjoint cases.

\textit{Case 1}. If $\pi\in \cq_{s,r}$ for some $s\ge 0$, we put $(\mu,\nu)=(\lambda,\pi)$. Notice that $\nu_1\le \mu_1+r$.

\textit{Case 2}. If $\pi_1=r+1$, then $\pi$ does not belong to any $\cq_{s,r}$. Notice also that $k\ge 1$. We now delete the largest part of $\lambda$ (which is $\lambda_1=k$) to get $\mu$, and add $\lambda_1$ to $\pi_1$ to get $\nu$. Then $\nu\in \cq_{1,r}$ and $\nu_1=k+r+1> \mu_1+r$. (See Fig.~\ref{fig:1}\textrm{(a)}.)

\textit{Case 3}. If $\pi$ has an $(s+r+1)\times s$ Durfee rectangle for some $s\ge 1$ and the largest part below the Durfee rectangle is $s+r+1$, then $1\le s< k$. We first delete the largest part of $\lambda$ (which is $\lambda_1=k$) to get $\mu$. Next we move each part on the right-hand side of the Durfee rectangle of $\pi$ one row below its original position. We now add $\lambda_1-s$ to the first row and $1$ to each of the next $s$ rows to get $\nu$. Then $\nu\in \cq_{s+1,r}$ and $\nu_1=k+r+1> \mu_1+r$. (See Fig.~\ref{fig:1}\textrm{(b)}.)

It is easy to check that the above process is invertible and hence we obtain a bijection. This finishes the proof.
\end{proof}

\begin{figure}[ht]
\caption{The bijection $\phi$ in the proof of Lemma \ref{le:1}}
\vspace{1em}
\label{fig:1}
\includegraphics[width=0.9\textwidth]{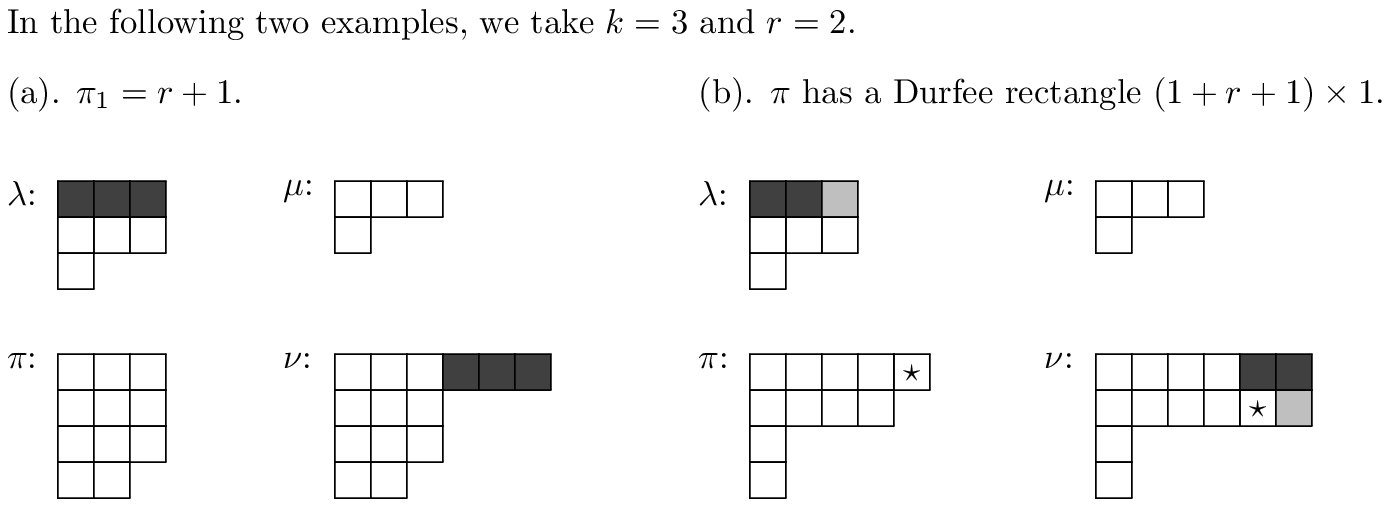}
\end{figure}

\begin{lemma}\label{le:2}
\begin{equation}
\sum_{(\lambda^*,\pi^*)\in \co_k^* \times \co_{k+r}^*} q^{|\lambda^*|+|\pi^*|}=\sum_{(\lambda,\pi)\in\cp_k\times \cp_{\le k+r}} q^{4|\lambda|+4|\pi|+2r}.
\end{equation}
\end{lemma}

\begin{proof}
Let $(\lambda,\pi)\in\cp_k\times \cp_{\le k+r}$. We split each square in the Young diagrams of $\lambda$ and $\pi$ into four squares and get $\lambda'$ and $\pi'$, both of which are even partitions where each different part appears an even number of times. We further notice that $\lambda'_1=2k$ and $\pi'_1\le 2k+2r$. We now delete one of the largest parts of $\lambda'$ and take its conjugate to obtain $\lambda^*$. On the other hand, we append $2k+2r$ to the top of $\pi'$ and take its conjugate to obtain $\pi^*$. Obviously, $\lambda^*\in \co_k^*$ and $\pi^* \in \co_{k+r}^*$. One may see Fig.~\ref{fig:2} for an example.

This process is invertible. Hence the lemma follows readily.
\end{proof}

\begin{figure}[ht]
\caption{An example in the proof of Lemma \ref{le:2} ($k=3$ and $r=1$)}
\vspace{1em}
\label{fig:2}
\includegraphics[width=0.9\textwidth]{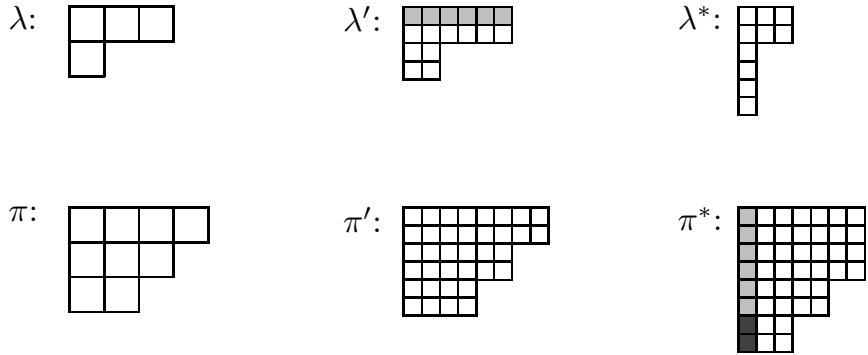}
\end{figure}

\begin{lemma}\label{le:3}
\begin{equation}
\sum_{(\mu^*,\nu^*)\in \ce^* \times \ceo_{k,r}^*} q^{|\mu^*|+|\nu^*|}=\sum_{(\mu,\nu)\in\cp \times \cq_{k,r}} q^{4|\mu|+4|\nu|+2r}.
\end{equation}
\end{lemma}

\begin{proof}
Let $(\mu,\nu)\in\cp \times \cq_{k,r}$.  We split each square in the Young diagrams of $\mu$ and $\nu$ into four squares and get $\mu'$ and $\nu'$. We first put $\mu^*=\mu'\in \ce^*$. For $\nu'$, we observe that it has a Durfee rectangle $(2k+2r)\times 2k$ while on the right-hand side of this Durfee rectangle, there are two columns of size $2k$. We now delete one of the two columns and append $2k+2r$ below the Durfee rectangle to obtain $\nu^*$. It is not difficult to check that $\nu^* \in \ceo_{k,r}^*$. One may see Fig.~\ref{fig:3} for an example.

This process is invertible. Hence the lemma follows readily.
\end{proof}

\begin{figure}[ht]
\caption{An example in the proof of Lemma \ref{le:3} ($k=3$ and $r=1$)}
\vspace{1em}
\label{fig:3}
\includegraphics[width=0.9\textwidth]{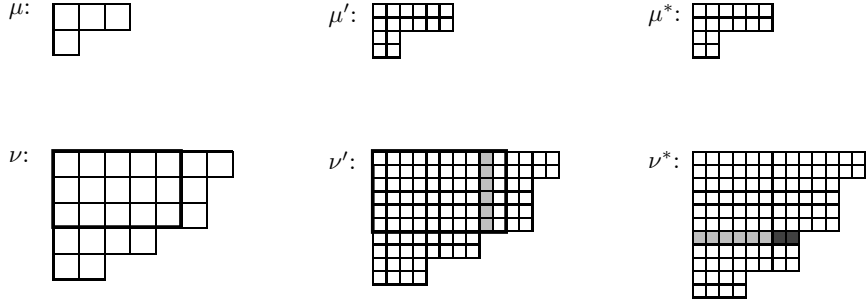}
\end{figure}

We now finish the combinatorial proof of \eqref{eq:gf2}.

\begin{proof}[Combinatorial proof of \eqref{eq:gf2}]
We rewrite \eqref{eq:gf2} as
$$\frac{1}{(q^4;q^4)_\infty}\sum_{\pi \in \ceo^*} z^{\eoc(\pi)}q^{|\pi|}=\frac{1}{(z^2q^2;q^4)_\infty(z^{-2}q^2;q^4)_\infty}.$$

We observe that given a partition in $\ceo^*$ with even-odd crank being $2r$, its conjugate is also in $\ceo^*$ and has even-odd crank $-2r$. By the symmetry, it suffices to prove
$$\sum_{(\lambda^*,\pi^*)\in \cup_{k\ge 0}\left(\co_k^* \times \co_{k+r}^*\right)} q^{|\lambda^*|+|\pi^*|} = \sum_{(\mu^*,\nu^*)\in \cup_{k\ge 0}\left(\ce^* \times \ceo_{k,r}^*\right)} q^{|\mu^*|+|\nu^*|}$$
for each nonnegative integer $r$. (Here $\nu^*$ runs through all partitions in $\ceo^*$ with even-odd crank being $2r$.)

Finally, we claim that the above identity is a direct consequence of Lemmas \ref{le:1}--\ref{le:3} and hence we complete the proof.
\end{proof}

\begin{remark}
It is worth pointing out that one may obtain a more direct bijection between $\cup_{k\ge 0}\left(\co_k^* \times \co_{k+r}^*\right)$ and $\cup_{k\ge 0}\left(\ce^* \times \ceo_{k,r}^*\right)$ by using a similar argument as the proof of Lemma \ref{le:1}. However, Lemma \ref{le:1} and its combinatorial proof have independent interest.
\end{remark}

\section{The largest even part}\label{sect:lep}

Let $\eo_0(n)$ (resp.~$\eo_2(n)$) denote the number of partitions of $n$ in $\ceo^*$ with largest even part congruent to $0$ (resp.~$2$). We are interested in the relation between $\eo_0(n)$ and $\eo_2(n)$.

Notice that
\begin{align*}
\sum_{n\ge 0}(\eo_0(n)-\eo_2(n))q^n &= \sum_{k\ge 0} \frac{(-q^2)^k}{(q^4;q^4)_k (q^{4k+2};q^4)_\infty}\\
&=\frac{1}{(q^2;q^4)_\infty}\sum_{k\ge 0} \frac{(q^2;q^4)_k (-q^2)^k}{(q^4;q^4)_k}\\
&= \frac{1}{(q^2;q^4)_\infty}\frac{(-q^4;q^4)_\infty}{(-q^2;q^4)_\infty}\\
&= \frac{(-q^4;q^4)_\infty}{(q^4;q^8)_\infty},
\end{align*}
where the second last identity comes from the celebrated $q$-binomial theorem (cf.~\cite[p.~17, Theorem 2.1]{And1976})
$$\sum_{n\ge 0}\frac{(a;q)_n z^n}{(q;q)_n}=\frac{(az;q)_\infty}{(z;q)_\infty}.$$

The trivial fact that $(-q^4;q^4)_\infty/(q^4;q^8)_\infty$ is a positive series of $q^4$ immediately leads to the following surprising result.

\begin{theorem}\label{th:r1}
We have
$$\eo_0(n)\begin{cases}
=\eo_2(n) & \text{if $n$ is not divisible by $4$};\\
>\eo_2(n) & \text{if $n$ is divisible by $4$}.
\end{cases}$$
\end{theorem}

\begin{example}
We have $\eo_0(6)=2$ since $6$ has partitions
$$3+3,\quad 1+1+1+1+1+1$$
and $\eo_2(6)=2$ since $6$ has partitions
$$6,\quad 2+2+2.$$
\end{example}

\begin{example}
We have $\eo_0(8)=4$ since $8$ has partitions
$$8,\quad 4+2+2,\quad 3+3+1+1,\quad 1+1+1+1+1+1+1+1$$
and $\eo_2(8)=1$ since $8$ has partition
$$3+2+2.$$
\end{example}

\section{A weighted overpartition generalization}\label{sect:overpartition}

The result in Sect.~\ref{sect:lep} is motivated by Ae Ja Yee's suggestion on an overpartition analog of Andrews' partition function.

Let $\overline{\ceo}^*$ be the set of overpartitions where all even parts are below odd parts and only the largest even part appears an odd number of times. For $\pi \in \overline{\ceo}^*$, we use $o(\pi)$ to count the number of overlined parts in $\pi$. It is not difficult to write the generating function
\begin{align*}
\sum_{\pi \in \overline{\ceo}^*} z^{o(\pi)}q^{|\pi|}&= \frac{(-zq^2;q^4)_\infty}{(q^2;q^4)_\infty}+\sum_{k\ge 1}\frac{(1+z)q^{2k} (-zq^4;q^4)_{k-1} (-zq^{4k+2};q^4)_\infty}{(q^4;q^4)_k (q^{4k+2};q^4)_\infty}\\
&= \frac{(-zq^2;q^4)_\infty}{(q^2;q^4)_\infty} \sum_{k\ge 0} \frac{(-z;q^4)_{k} (q^2;q^4)_k q^{2k}}{(q^4;q^4)_k (-zq^2;q^4)_k}.
\end{align*}
However, this generating function seems to be unable to get further simplified.

On the other hand, the Bailey--Daum sum (cf.~\cite[Eq.~(17.6.5)]{And2010}) tells
$$\sum_{n\ge 0}\frac{(a;q)_n (b;q)_n}{(q;q)_n (aq/b;q)_n}\left(-\frac{q}{b}\right)^n = \frac{(-q;q)_\infty(aq;q^2)_\infty(aq^2/b^2;q^2)_\infty}{(-q/b;q)_\infty(aq/b;q)_\infty}.$$
This motivates us to assign the following weight to $\pi \in \overline{\ceo}^*$
$$w(\pi)=\begin{cases}
1 & \text{if the largest even part of $\pi$ is divisible by $4$};\\
-1 & \text{if the largest even part of $\pi$ is not divisible by $4$}.
\end{cases}$$

\begin{theorem}
\begin{equation}\label{eq:gf-over}
\sum_{\pi \in \overline{\ceo}^*} w(\pi) z^{o(\pi)}q^{|\pi|} = \frac{(-q^4;q^4)_\infty(-zq^4;q^8)_\infty^2}{(q^4;q^8)_\infty}.
\end{equation}
\end{theorem}

\begin{proof}
We have
\begin{align*}
\sum_{\pi \in \overline{\ceo}^*} w(\pi) z^{o(\pi)}q^{|\pi|}&= \frac{(-zq^2;q^4)_\infty}{(q^2;q^4)_\infty}+\sum_{k\ge 1}\frac{(1+z)(-q^2)^k (-zq^4;q^4)_{k-1} (-zq^{4k+2};q^4)_\infty}{(q^4;q^4)_k (q^{4k+2};q^4)_\infty}\\
&= \frac{(-zq^2;q^4)_\infty}{(q^2;q^4)_\infty} \sum_{k\ge 0} \frac{(-z;q^4)_{k} (q^2;q^4)_k (-q^2)^k}{(q^4;q^4)_k (-zq^2;q^4)_k}\\
&=  \frac{(-zq^2;q^4)_\infty}{(q^2;q^4)_\infty}\frac{(-q^4;q^4)_\infty(-zq^4;q^8)_\infty^2}{(-q^2;q^4)_\infty (-zq^2;q^4)_\infty}\\
&= \frac{(-q^4;q^4)_\infty(-zq^4;q^8)_\infty^2}{(q^4;q^8)_\infty}.
\end{align*}
\end{proof}

If we take $z=0$, then we obtain the generating function identity of $\eo_0(n)-\eo_2(n)$ in Sect.~\ref{sect:lep}. If we take $z=1$, then
\begin{align*}
\sum_{\pi \in \overline{\ceo}^*} w(\pi) q^{|\pi|}= \frac{(-q^4;q^4)_\infty(-q^4;q^8)_\infty^2}{(q^4;q^8)_\infty}.
\end{align*}
Again we notice that this is a positive series of $q^4$. Let $\oeo_0(n)$ (resp.~$\oeo_2(n)$) denote the number of overpartitions of $n$ in $\overline{\ceo}^*$ with largest even part congruent to $0$ (resp.~$2$).

\begin{theorem}\label{th:r2}
We have
$$\oeo_0(n)\begin{cases}
=\oeo_2(n) & \text{if $n$ is not divisible by $4$};\\
>\oeo_2(n) & \text{if $n$ is divisible by $4$}.
\end{cases}$$
\end{theorem}

\section{Conclusion}

The combinatorial proof of \eqref{eq:gf-over} seems to be more difficult than Andrews' non-overlined version and hence we cry out for such an interpretation. On the other hand, it would be interesting to see direct combinatorial proofs of Theorems \ref{th:r1} and \ref{th:r2}.

\subsection*{Acknowledgements}

I would like to thank George E. Andrews for many helpful discussions. I would also like to thank Ae Ja Yee for her suggestion on the overpartition generalization in Sect.~\ref{sect:overpartition}.

\bibliographystyle{amsplain}

\end{document}